\documentclass[12pt, centertags]{amsart}
\usepackage{amsmath, amstext, amsthm, a4, amssymb, amscd}
\usepackage{amssymb, graphics, epsfig, color}
\usepackage[mathscr]{eucal}
\usepackage[all]{xy}
\usepackage{mathrsfs}
\usepackage{graphicx}
\makeindex

\numberwithin{equation}{section}

 \newcommand{\field}[1]{\mathbb{#1}}
 \newcommand{\Z}{\field{Z}} \newcommand{\R}{\field{R}}

 \def\Im{{\rm Im}} 
 \newcommand{\til}[1]{\widetilde{#1}} 
    
    \newcommand{\tD}{\til{D}}

  \DeclareMathOperator{\rk}{rk} \DeclareMathOperator{\Id}{Id} 
 \DeclareMathOperator{\tr}{Tr} \DeclareMathOperator{\str}{Tr_{s}}

 \DeclareMathOperator{\End}{End}

   \newcommand{\om}{\omega}
 
 \newcommand{\iprod}[1]{\langle #1\rangle }

 \newtheorem{thm}{Theorem}[section] \newtheorem{lemma}[thm]{Lemma}  \newtheorem{prop}[thm]{Proposition} 
 
  \theoremstyle{definition}  \theoremstyle{definition}
 \newtheorem{defn}[thm]{Definition} %
 \theoremstyle{remark} %
 \newcommand{\be}{\begin{eqnarray}}
   
 \newcommand{\comment}[1]{}
\newcommand{\onehalf}{\frac{1}{2}}
\newcommand{\Cinf}[1]{C^{\infty}\left( #1\right) }

\newcommand{\ra}{\rightarrow}
\newcommand{\lra}{\longrightarrow}
\newcommand{\rac}[1]{\overset{#1}\ra}

\newcommand{\n}{\nabla}
\newcommand{\set}[1]{\{#1\}}

\def\cH{\mathscr{H}}
\def\cT{\mathscr{T}}

\def\cE{\mathscr{E}}

\newcommand{\onefour}{\frac{1}{4}}

\newcommand{\G}{\Gamma}
\newcommand{\nF}{\n^{F}}
\newcommand{\hF}{h^{F}}

\newcommand{\gTY}{g^{TY}}
\newcommand{\gTZ}{g^{TZ}}
\newcommand{\paru}{\frac{\partial}{\partial u}}
\newcommand{\TMgh}{\mathscr{T}_{\rm r}(T^HM,g^{TZ},h^F)}
\newcommand{\TMgha}{\mathscr{T}_{\rm a}(T^HM_1,g^{TZ_1},h^F)}
\newcommand{\TMghb}{\mathscr{T}_{\rm a, r}(T^HM_2,g^{TZ_2},h^F)}
\newcommand{\TMghc}{\mathscr{T}_{\rm a,a}(T^HM_3,g^{TZ_3},h^F)}
\newcommand{\OZF}{\Omega^\bullet(Z,F)}
\newcommand{\OZbF}{\Omega^\bullet(Z_b,F)}
\newcommand{\OMF}{\Omega^\bullet(M,F)}
\newcommand{\oTZ}{o(TZ)}\newcommand{\dvZ}{dv_Z}
\newcommand{\hOZF}{h^{\OZF}}
\newcommand{\PTZ}{P^{TZ}}
\newcommand{\nOZF}{\n^{\OZF}}
\newcommand{\nOZFa}{\n^{\OZF*}}
\newcommand{\dMa}{d^{M*}}
\newcommand{\iTa}{i_T^*}
\newcommand{\dZa}{d^{Z*}}
\newcommand{\en}{e_{\mathfrak{n}}}
\newcommand{\enu}{e^{\mathfrak{n}}}
\newcommand{\parxm}[1]{\frac{\partial #1}{\partial x_m}}
\newcommand{\dxm}{dx_m}\newcommand{\dxmu}{(dx_m)^*}
\newcommand{\ear}[3]{e_{a,r}^{#1 \frac{\partial^2}{\partial x^2_m}}(#2,#3)}
\newcommand{\eaa}[3]{e_{a,a}^{#1 \frac{\partial^2}{\partial x^2_m}}(#2,#3)}
\newcommand{\eDD}[3]{e_{D,D}^{#1 \frac{\partial^2}{\partial x^2_m}}(#2,#3)}
\newcommand{\eNN}[3]{e_{N,N}^{#1 \frac{\partial^2}{\partial x^2_m}}(#2,#3)}
\newcommand{\eND}[3]{e_{N,D}^{#1 \frac{\partial^2}{\partial x^2_m}}(#2,#3)}
\newcommand{\eDN}[3]{e_{D,N}^{#1 \frac{\partial^2}{\partial x^2_m}}(#2,#3)}
\newcommand{\ak}{\alpha_k}
\newcommand{\zpDN}{\zeta'_{D,N}}\newcommand{\zpDD}{\zeta'_{D,D}}
\newcommand{\zDN}{\zeta_{D,N}}\newcommand{\zDD}{\zeta_{D,D}}
\newcommand{\intzi}{\int_{0}^{\infty}}
\newcommand{\tzeta}{\widetilde{\zeta}}
\newcommand{\Xmoo}{X_{[-1,1]}} \newcommand{\Ymoo}{Y_{[-1,1]}}
\newcommand{\TMz}{\mathscr{T}(T^HM,g^{TZ},h^F)}
\newcommand{\TrMa}{\mathscr{T}_{\rm r}(T^HM_1,g^{TZ_1},h^F)}
\newcommand{\TaMa}{\mathscr{T}_{\rm a}(T^HM_1,g^{TZ_1},h^F)}
\newcommand{\TaMb}{\mathscr{T}_{\rm a}(T^HM_2,g^{TZ_2},h^F)}
\newcommand{\TaaMc}{\mathscr{T}_{\rm a,a}(T^HM_3,g^{TZ_3},h^F)}
\newcommand{\cTX}{\mathscr{T}(T^HX,g^{TY},h^F)}

\begin{document}

 \title{A comparison of the absolute and relative real analytic torsion forms}
 \date{\today}
\author{Jialin Zhu}
 \address{Mathematical Science Research Center, Chongqing University of Technology, No. 69 Hongguang Road, Chongqing 400054, China}
 \email{jialinzhu@cqut.edu.cn}

\begin{abstract}
In this paper we establish a comparison formula of the absolute and relative real analytic torsion forms over fibrations with boundaries. The key tool is a gluing formula of analytic torsion forms proved by M. Puchol, Y. Zhang and the author. As a consequence of the comparison formula, we prove another version of the gluing formula of the analytic torsion forms conjectured originally by the author.
\end{abstract}

\maketitle

\tableofcontents
\setcounter{section}{-1}

\section{Introduction}\label{s0}

In this paper, we first prove a comparison formula of the absolute and relative Bismut-Lott real analytic torsion forms by using the gluing formula of the Bismut-Lott torsion forms established recently by M. Puchol, Y. Zhang and the author \cite{PZZ2}. As a consequence of the comparison formula, we prove another version of the gluing formula of the Bismut-Lott analytic torsion forms conjectured originally by the author in \cite[Conjecture 1.1]{Zhu15}

 As the analytic analogue of the Reidemeister topological torsion \cite{Mil66}, Ray
 and Singer \cite{RaySing71} introduced the Ray-Singer analytic torsion associated to de Rham complex twisted by a flat vector bundle $F$ over a
compact oriented Riemannian
manifold $M$. They also conjectured that the topological torsion and analytic torsion coincide,
proved by Cheeger \cite{Chg} and M\"{u}ller \cite{Mu78} independently in the unitary case.
Bismut and Zhang \cite{BZ92} and M\"{u}ller \cite{Mu93} simultaneously considered its generalizations. M\"{u}ller extended
his result to the case where the dimension of the manifold is odd and only the metric induced on $\det{F}$ is required to be flat.
Bismut and Zhang generalized the original Cheeger-M\"{u}ller theorem to arbitrary flat vector bundles with arbitrary Hermitian
metrics.

In \cite{BL}, Bismut and Lott generalized the Ray-Singer torsion to smooth fibrations. They first proved a family index theorem for flat vector bundles, which relates the odd characteristic classes of a flat complex vector bundle on the total
 space of the fibration to those of its direct image on the base. Then they improved their theorem to the level of differential forms, where the Bismut-Lott analytic torsion forms appear as the transgression terms.
 Inspired by \cite{BL},
Igusa and Klein
\cite{Igusa02} extended the Reidemeister torsion to the family case and obtained the Igusa-Klein
topological torsion.

L\"{u}ck \cite{Luck93} established the gluing formula for the analytic torsion for unitary flat vector bundles when the Riemannian metric has product structure near the boundary by using the results in \cite{LoRo91}.
 There are also other works on the gluing problem of the
 analytic torsion (cf. \cite{Has},  \cite{Vishik95}). Finally,
Br\"{u}ning and Ma \cite{BruMa06} established the anomaly formula of the analytic torsion on manifolds with boundary,
then they \cite{BM13} proved the gluing formula of analytic torsion for any flat vector bundles and without any assumptions on the product structures near the boundary. In \cite{BM13}, Br\"{u}ning and Ma essentially used the equivariant Bismut-Zhang theorem \cite{BisZh94}. In \cite{PZZ21},  We give a pure
analytic proof of Br\"{u}ning-Ma's gluing formula while the product structures of metrics are assumed.

The gluing formula of Bismut-Lott torsions was first established in two important cases by the author in \cite{Zhu15}, \cite{Zhu16}. In \cite{Zhu15}, the gluing formula was proved under the assumption that there exists fiberwise Morse function. The proof is a generalization of Br\"uning-Ma's proof in \cite{BM13} to family case, based on the work of Bismut-Goette \cite{BGo}. In \cite{Zhu16}, the gluing formula was established by using adiabatic limit under the assumption that the fiberwise Dirac operator of the boundary is invertible. Recently, we proved the gluing formula in the general case \cite{PZZ2} by using adiabatic limit, scattering theory and a Witten-type deformation. Our gluing formula plays a key role in establishing the equivalence between Bismut-Lott torsion and Igusa-Klein torsion under Igusa's axiomatization \cite{Igu08} of higher torsion invariants \cite{PZZ3}.

Let $M\rac{\pi}S$ be a smooth fibration over $S$ with a compact $m$-dimensional fiber $Z$. We suppose that $X$ is a compact hypersurface in $M$ such that $M=M_{1}\cup_{X} M_{2}$ and $M_{1}, M_{2}$ are
 manifolds with the common boundary $X$. We also assume that
$$
Z_{1}\rightarrow M_{1}\overset{\pi}\rightarrow S, \quad Z_{2}\rightarrow M_{2}\overset{\pi}\rightarrow S, \quad \text{and } Y\rightarrow X\rac{\pi_{\partial}} S
$$
are all smooth fibrations with fiber $Z_{1, b}$, $Z_{2, b}$ and $Y_{b}$ at $b\in S$ such that
\begin{align}\begin{aligned}\label{e.359}
Z_{b}=Z_{1, b}\cup_{Y_{b}}Z_{2, b}.
\end{aligned}\end{align}
In other words, the fibrations $M_{1}$ and $M_{2}$ can be glued into $M$ along $X$ (see Figure \ref{figure2}).

Let $\big(H(Z_{2}, F),\, \nabla^{(Z_{2}; F)}\big)$ (resp. $\big(H(Z_{1}, Y;F),\, \nabla^{H(Z_{1}, Y, F)}\big)$) denote the flat vector bundle on $S$, whose fiber at $b\in S$ is isomorphic to the absolute (resp. relative) cohomology group
$H^{p}(Z_{2, b}, F)$ (resp. $H(Z_{1}, Y;F)$) of fiberwise de Rham complex \cite[\S III.(f)]{BL}.  Then we have a long exact sequence $(\mathscr{H},v)$ of flat vector bundles of cohomology groups (cf. \cite[(0.16)]{BM13}), i.e.,
\begin{align}\begin{aligned}\label{e.434}
\cH:\quad\cdots \longrightarrow H^{p}(Z, F)\overset{\beta_p}{\longrightarrow} H^{p}(Z_{2}, F) \overset{\delta_p}{\longrightarrow}
H^{p+1}(Z_{1}, Y, F) \overset{\alpha_{p+1}}{\longrightarrow}\cdots,
\end{aligned}\end{align}
where $\alpha:Z\rightarrow (Z,Z_{2})$ and $\beta:Z_{2}\rightarrow Z$ are natural maps.
The $\mathbb{Z}-$grading of (\ref{e.434}) on $H^{p}(Z,F)$, $H^{p}(Z_{2},F)$ and $H^{p}(Z_{1},Y,F)$ are given by $3p+1$, $3p+2$ and $3p$ respectively with $0\leq p \leq m$. Let $h^{\mathscr{H}}_{L^{2}}$ be the $L^{2}-$metric on each component of $\mathscr{H}$ induced by the Hodge theory.
Let $\nabla^{\mathscr{H}}$ denote the canonical flat connection on $\cH$.
Then a torsion form $\cT_{\cH}$ is associated with
the triple $(\mathscr{H}, \nabla^{\mathscr{H}},h^{\mathscr{H}}_{L^{2}})$ as in \cite[\S II]{BL}.

\begin{figure}
  \centering
  \includegraphics[width=10cm]{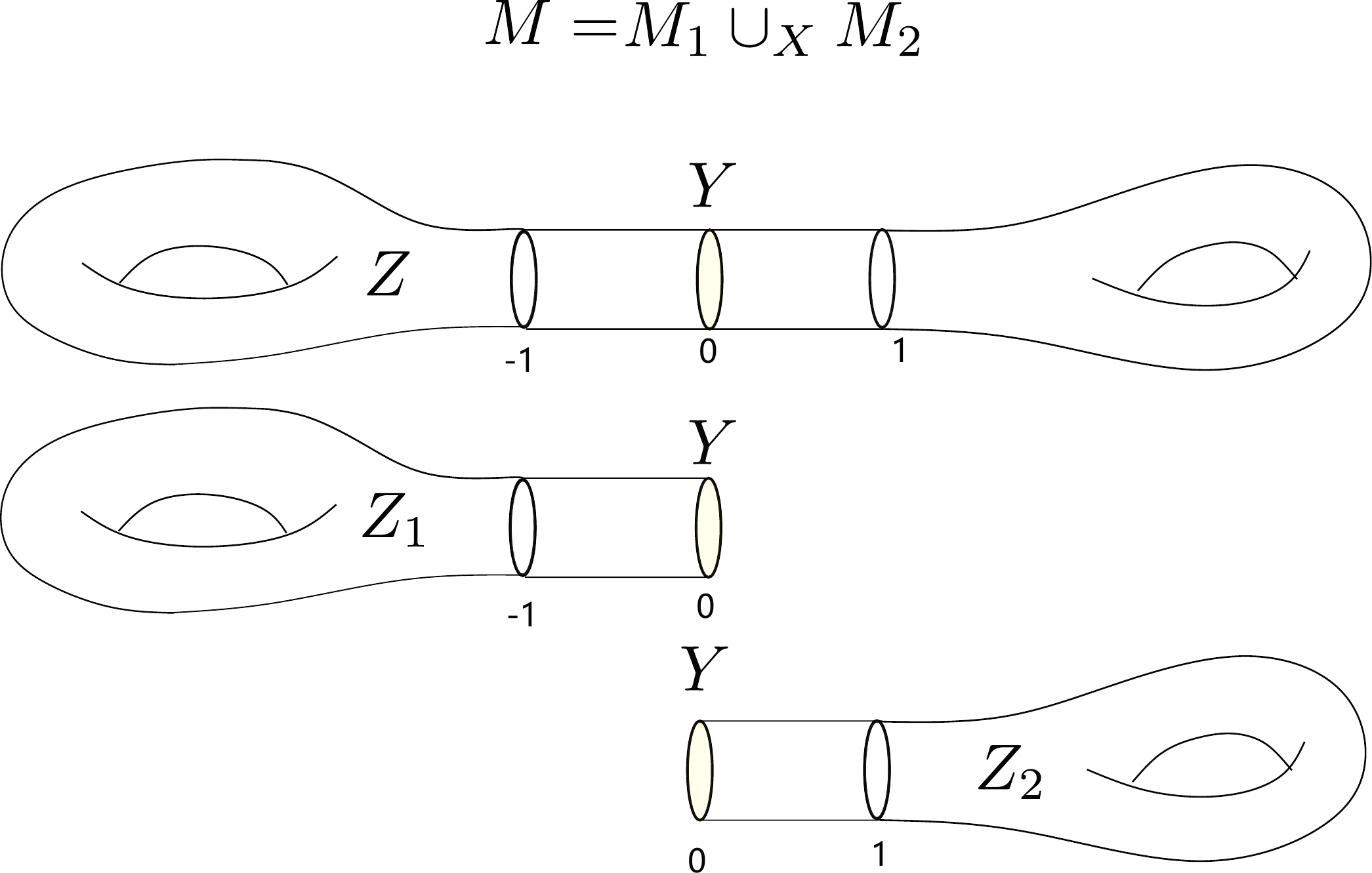}\\
\caption{Gluing relations}\label{figure2}
\end{figure}

Let $\cT_{a/r}(T^HM_1,g^{TZ_1},h^F)$ be the Bismut-Lott torsion forms with the absolute/relative boundary conditions (see \eqref{e.21}). The following theorem proved in Theorem \ref{t.1} is a comparison formula of the absolute/relative Bismut-Lott torsion forms.
\begin{thm}\label{t.3} Under assumptions of product structures \eqref{e.1}-\eqref{e.3}, we have
  \begin{align}\begin{aligned}\label{e.43}
&\cT_r(T^HM_1,g^{TZ_1},h^F)-\TMgha +\mathscr{T}_{\mathscr{H}'}\\
&\qquad\qquad\doteq-\cT(T^HX,g^{TY},h^F)+
 \frac{3}{2}\log 2\rk(F)\chi(Y),
\end{aligned}\end{align}
where ``$\doteq$'' means that the equality holds modulo some exact differential forms over $S$ and $\mathscr{T}_{\mathscr{H}'}$ is the torsion form associated with
  \begin{align}\begin{aligned}\label{e.44}
\cH':\quad\cdots \ra      H^k(Z_1,Y;F) \ra H^k(Z_1;F) \ra H^k(Y;F)\ra \cdots.
\end{aligned}\end{align}
\end{thm}

The following version of the gluing formula was originally conjectured by the author in \cite{Zhu15} under the setting of Figure \ref{figure2}.

\begin{thm}\label{t.2} Under the assumptions of product structures \eqref{e.1}-\eqref{e.3},
the following identity holds,
\begin{align}\begin{aligned}\label{e.355}
\mathscr{T}(T^{H}M, g^{TZ}, h^{F})-\mathscr{T}_{\rm r}(T^{H}M_{1}&,g^{TZ_{1}}, h^{F})
-\mathscr{T}_{\rm a}(T^{H}M_{2}, g^{TZ_{2}}, h^{F})\\
&\doteq \cT_{\cH}+\frac{\log 2}{2}\rk(F)\chi(Y),
\end{aligned}\end{align}
where $\cT_{\cH}$ is the torsion forms associated with the long exact sequence \eqref{e.434}.
\end{thm}

For the Ray-Singer analytic torsions, the above gluing formula was proved in \cite[(0.22)]{BM13} without any product structures of the metrics.

The whole paper is organised as follows. In Section \ref{s1}, we introduce some preliminary on the geometry over fibrations with boundaries. We also introduce the Bismut-Lott analytic torsion forms with the absolute/relative boundary conditions. In Section \ref{s2}, we establish Theorem \ref{t.3} the comparison formula of the absolute and relative torsion forms. In Section \ref{s3}, we compare the two versions of the gluing formulas and establish Theorem \ref{t.2}.

\textbf{Acknowledgments.} The author thanks M. Puchol and Y. Zhang for very useful and inspiring conversations.
This work was completed during the author's visit to the Institute of Geometry and Physics at the University of Science and Technology of China. He would like to thank IGP for hospitality.

\section{Fibrations with cylinder ends}\label{s1}
\subsection{Product structures}
Let $\pi:M\ra S$ be a smooth fibration with boundary $X:=\partial M$ and a compact manifold $Z$ as its standard fiber. We assume that the boundary $X$ of $M$ is a smooth fibration denoted by $\pi_{\partial}: X\ra S$ with fiber $Y$ such that $Y=\partial Z$.

\begin{defn}\label{d.1} For a compact manifold $X$ and a subset $I$ of $\R$, we set $X_I:=X\times I$, for example $X_{[-a,a]}=X\times [-a,a]$ ($a>0$), $X_{\R}=X\times (-\infty,+\infty)$.
\end{defn}

Let $TM$ be the tangent bundle of $M$. Let $TZ$ be the vertical subbundle of $TM$. Let $T^HM$ be the horizontal subbundle of $TM$, such that $TM=TZ\oplus T^HM$. Let $TY$ be the vertical tangent bundle of the fibration $X$, which is a subbundle of $TZ|_X$. Let $N$ be the normal bundle of $X\subset M$, i.e., $N:=TM/TX$, then by our assumption we have $TZ/TY\cong N$. In this case, $N\ra X$ is a trivial line bundle.

Let $g^{TZ}$ be a metric on $TZ$ and $g^{TY}$ be the metric on $TY$ induced by $\gTZ$. We identity $N$ with the orthogonal complement of $TY$ in $TZ$ with respect to $\gTZ$, thus we have $TZ|_X=TY\oplus N$.

Assume that $X_{[-1,2]}$ is a product neighborhood of $X\subset M$, and $\partial M$ is identified with $X_{\set{2}}$. Let $\psi: X\times [-1,2]\ra X$ be the projection on the first factor. We assume that
$T^HM$ and $\gTZ$ have product structures on $X_{[-1,2]}$, i.e.,
\begin{align}\begin{aligned}\label{e.1}
(T^HM)|_X\subset TX,\quad (T^HM)|_{X_{[-1,2]}}=\psi^*((T^HM)|_X),
\end{aligned}\end{align}
\begin{align}\begin{aligned}\label{e.2}
\gTZ|_{(x',x_m)}=\gTY(x')+dx^2_m,\quad (x',x_m)\in X\times [-1,2].
\end{aligned}\end{align}
Then $T^HM:=(T^HM)|_X$ gives a horizontal bundle of fibration $X$, such that $TX=T^HX\oplus TY$.

Let $(F,\nF)$ be a flat complex vector bundle on $M$ with a flat connection $\nF$, i.e., $(\nF)^2=0$. We trivialize $F$ along the $x_m-$direction, by using the parallel transport with respect to $\nF$, then we have $(F,\nF)|_{X_{[-1,2]}}=\phi^*\left(F|_X,\nF|_X\right)$. Under this trivialization, we assume
\begin{align}\begin{aligned}\label{e.3}
h^F|_{X_{[-1,2]}}=\psi^*(h^F|_X).
\end{aligned}\end{align}

\subsection{Bismut's superconnection}
Let $\OZF$ be the infinite-dimensional $\Z-$graded vector bundle over $S$ whose fiber is $\OZbF$ at $b\in S$. Then we have
\begin{align}\begin{aligned}\label{e.12}
\OMF=\Omega^\bullet(S,\OZF).
\end{aligned}\end{align}

 The fiberwise Riemannian volume form $\dvZ$ associated with $\gTZ$ is a section of $\Lambda^m(T^*Z)\otimes \oTZ$ over $M$, where $\oTZ$ is the orientation line bundle of $TZ$. The Hermitian metric on the infinite dimensional bundle $\OZF\ra S$ induced by $\gTZ$ and $\hF$ is defined as: for $s,\,s'\in \OZbF$, $b\in S$,
\begin{align}\begin{aligned}\label{e.13}
\iprod{s,s'}_{\hOZF}(b):=\int_{Z_b}\iprod{s,s'}_{g^{\Lambda(T^*Z)\otimes F}}(x)dv_{Z_b}.
\end{aligned}\end{align}

Let $\PTZ:TM\ra TZ$ be the natural projection. For $U\in TS$, let $U^H$ be the horizontal lift of $U$ in $T^HM$, so that $\pi_*U^H=U$.

\begin{defn}
  For $U\in TS$, the Lie derivative of $U^H$ defines a connection on $\OZF$,
\begin{align}\begin{aligned}\label{e.14}
\nOZF_Us:=L_{U^H}s,\quad \text{for} \quad s\in \Cinf{S,\OZF}.
\end{aligned}\end{align}
\end{defn}

Let $d^Z$ be the exterior differentiation twisted by $\nF$ on the fibers $Z$. For $U_1,U_2\in TS$, set
\begin{align}\begin{aligned}\label{e.15}
T(U_1,U_2)=-\PTZ[U^H_1,U^H_2]\in \Cinf{M,TZ},
\end{aligned}\end{align}
which defines a tenso $T\in \Cinf{M,\pi^*(\Lambda^2(T^*S))\otimes TZ}$. Let $i_T$ be the interior multiplication by $T$ in the vertical direction.
The exterior differentiation $d^M$ twisted by $\nF$ acting on $\OMF$ can be decomposed as:
\begin{align}\begin{aligned}\label{e.16}
d^M=d^Z+\nOZF+i_T.
\end{aligned}\end{align}

Let $\nOZFa,\,\dMa,\,\iTa,\,\dZa$ be the formal adjoints of $\nOZF,\,d^M,\,i_T,\,d^Z$ with respect to $\hOZF$ in the senses of \cite[Def. 1.6]{BL}. Then by \eqref{e.16} we have $\dMa=\dZa+\nOZFa+\iTa$.
Let $N_Z$ be the number operator on $\OZF$, i.e., it acts by multiplication by $k$ on $\Omega^k(Z,F)$.
For $t>0$, we set (cf. \cite[(3.50)]{BL})
\begin{align}\begin{aligned}\label{e.17}
C_t&:=\onehalf\left(t^{-N_Z/2}\dMa t^{N_Z/2}+t^{N_Z/2}d^M t^{-N_Z/2}\right),\\
D_t&:=\onehalf\left(t^{-N_Z/2}\dMa t^{N_Z/2}-t^{N_Z/2}d^M t^{-N_Z/2}\right),
\end{aligned}\end{align}
where $D_t$ is an odd element of $\Omega(S,\End(\OZF))$ and $C_t$ is essentially the same as the Bismut superconnection introduced by Bismut in \cite{Bismut86}.

\subsection{Absolute and relative boundary conditions}
Let $\en$ be the inward-pointing unit normal vector field on $X\subset M$, and $\enu$ be its dual vector field, then we extend $\en, \enu$ on $X_{[-1,2]}$. By the product structure \eqref{e.2}, we have
$\en=-\parxm{},\,\enu=-\dxm$ on $X_{[-1,2]}$. For $\sigma\in \OMF$, we say that $\sigma$ satisfies the \textbf{absolute boundary conditions}, if
\begin{align}\begin{aligned}\label{e.18}
(i_{\en}\sigma)|_X=(i_{\en}d^Z\sigma)|_X=0.
\end{aligned}\end{align}
We say that $\sigma$ satisfies the \textbf{relative boundary conditions}, if
\begin{align}\begin{aligned}\label{e.19}
(\enu\wedge\sigma)|_X=(\enu\wedge d^Z\sigma)|_X=0.
\end{aligned}\end{align}

Let $\Omega_{a/r}(Z,F)$ be the space of differential forms verifying the absolute/relative boundary conditions respectively. Let $H_{a/r}(Z,F)$ be the flat vector bundle of fiberwise absolute/relative cohomology groups with canonical connection $\n^{H_{a/r}(Z,F)}$. Set
\begin{align}\begin{aligned}\label{e.20}
&\chi_{a/r}(Z)=\sum_{p=0}^m(-1)^p \rk H^p_{a/r}(Z),
\quad\chi'_{a/r}(Z,F)=\sum_{p=0}^m(-1)^p p\rk H^p_{a/r}(Z,F).
\end{aligned}\end{align}
Let $\varphi:\Omega(S)\ra\Omega(S)$ be the linear map such that for all homogeneous $\omega\in \Omega(S)$, $\varphi \om=(2i\pi)^{-(\deg \om)/2}\om$. Let $f(x)=xe^{x^2}$ and $f'(x)=(1+2x^2)e^{x^2}$.  For $t>0$, we set
\begin{align}\begin{aligned}\label{e.22}
f^{\wedge}(D_t,\hOZF)_{a/r}:=\varphi\str\left[\frac{N}{2}f'(D_t)_{a/r}\right].
\end{aligned}\end{align}

The following definition of Bismut-Lott analytic torsion forms is given by \cite[Def. 3.22]{BL} without boundaries and by \cite[Def. 2.18]{Zhu15} with boundaries.

\begin{defn}\label{d.2}
The Bismut-Lott torsion forms with absolute/relative boundary conditions are defined as
\begin{align}\begin{aligned}\label{e.21}
\cT_{a/r}(T^HM,\gTZ,\hF)&=-\int^{+\infty}_0\left[f^{\wedge}(D_t,\hOZF)_{a/r}
-\frac{\chi'_{a/r}(Z,F)}{2}f'(0)\right.\\
&\quad\quad \left.-\Big(\onefour m\rk(F)\chi_{a/r}(Z)-\frac{\chi'_{a/r}(Z,F)}{2}\Big)f'(\frac{i\sqrt{t}}{2})\right]\frac{dt}{t}.
\end{aligned}\end{align}
\end{defn}

\section{Comparison formula of the absolute and relative Bismut-Lott torsion forms}\label{s2}

\subsection{A consequence of the gluing formula}

\begin{figure}
  \centering
  \includegraphics[width=10cm]{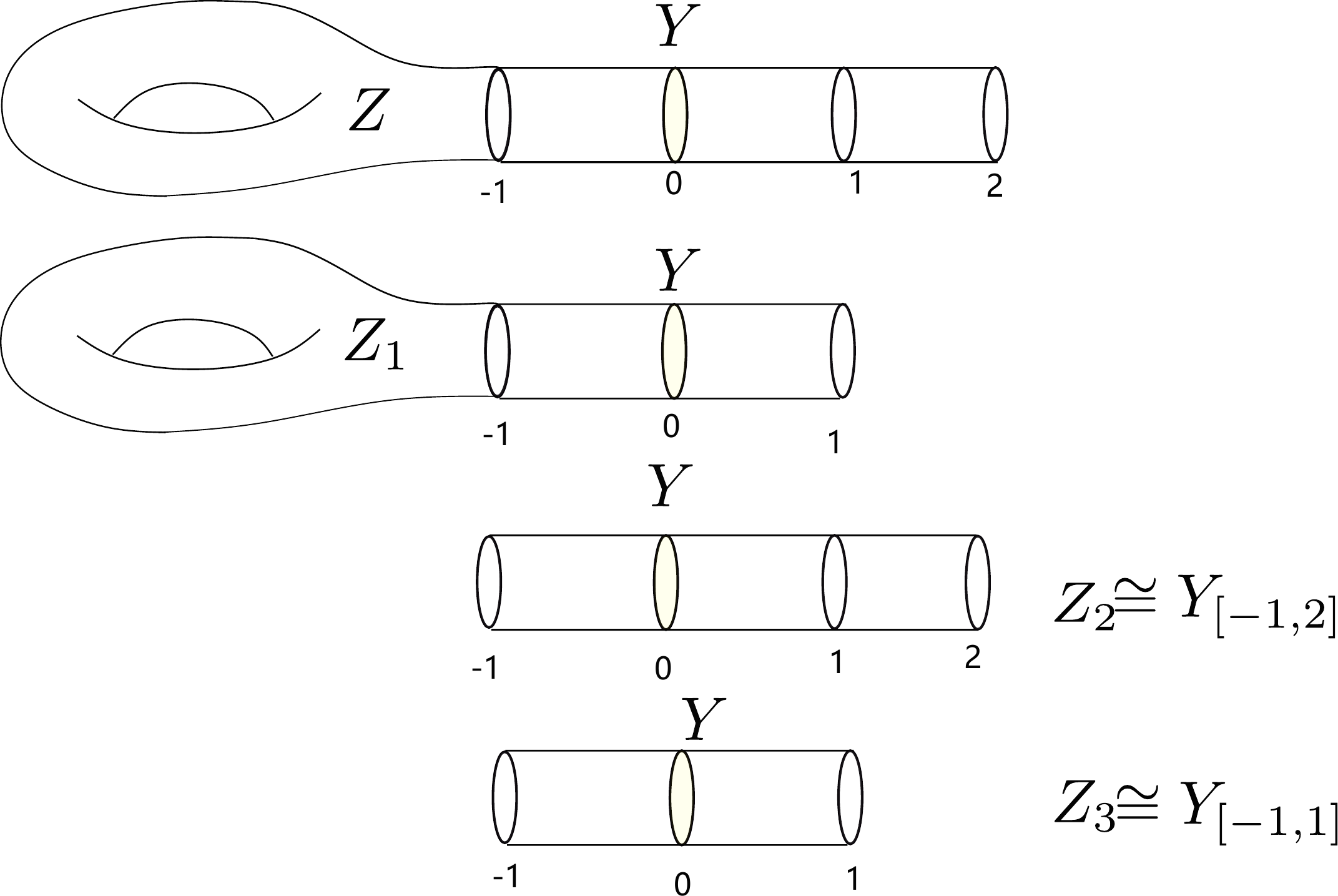}\\
\caption{Gluing relations}\label{figure1}
\end{figure}
As in Figure \ref{figure1}, let $\pi: M_1\ra S$ be a fibration obtained by cutting $X_{[1,2]}$ from $M$. Set $M_2=X\times [-1,2]$ and $M_3=X\times [-1,1]$. The standard fibers of $M_1,\,M_2,\,M_3$ are denoted respectively by $Z_1,\,Z_2\cong Y_{[-1,2]},\,Z_3\cong Y_{[-1,1]}$. The flat vector bundles and the metrics on $M_1,\,M_2,\,M_3$ are all induced from $M$ in an obvious way.

We impose the relative boundary conditions at $X_{\set{2}}\subset M$, $X_{\set{2}}\subset M_2$ and the absolute boundary conditions at $X_{\set{1}}\subset M_1$, $X_{\set{-1}}\subset M_2$, $X_{\set{\pm 1}}\subset M_3$. Let $\TMgh$, $\TMgha,\,\TMghb,\,\TMghc$ be the corresponding Bismut-Lott real analytic torsion forms associated with $M$, $M_1$, $M_2$ and $M_3$ respectively. By applying Theorem 0.1 of \cite{PZZ2} under the setting of Figure \ref{figure1}, we obtain
\begin{align}\begin{aligned}\label{e.4}
&\TMgh-\TMgha\\
&\qquad\qquad-\TMghb+\TMghc+\mathscr{T}_{\mathscr{H}}\doteq0,
\end{aligned}\end{align}
where ``$\doteq$'' means the equality holds modulo some exact differential forms and $\mathscr{T}_{\mathscr{H}}$ is the torsion form associated to the following flat exact sequence of cohomology bundles,
\begin{align}\begin{aligned}\label{e.5}
\mathscr{H}:\quad\cdots \ra H^k(Z,Y_{\set{2}};F)\ra H^k(Z_1;F)\oplus H^k(Z_2,Y_{\set{2}};F)\ra H^k(Z_3;F)\ra\cdots.
\end{aligned}\end{align}

\textbf{Remark}: Since in the proof of the gluing formula in \cite{PZZ2} all the operations are localized near the cutting hypersurface, the formula is also suitable for the setting of Figure \ref{figure1} where the total fibration $M$ has a boundary.\\

Let $\phi: (-1,1)\ra (-1,2)$ be a smooth monotonic bijective function, such that the function $(\phi'-1)$ has compact support in $(-1,1)$. In an obvious way, $\phi$ can be extended to be functions $\phi:X_{[-1,1]}\ra  X_{[-1,2]}$ and $\phi:M_1\ra M$. Then we have
\begin{align}\begin{aligned}\label{e.6}
\TMgh&\doteq\cT_r(T^HM_1,\phi^*g^{TZ},h^F),\\
\TMghb&\doteq\cT_{a,r}(T^HM_3,\phi^*g^{TZ_2},h^F).
\end{aligned}\end{align}
By \eqref{e.4}, \eqref{e.6}, we get
\begin{align}\begin{aligned}\label{e.7}
&\cT_r(T^HM_1,\phi^*g^{TZ},h^F)-\TMgha\\
&\qquad\qquad-\cT_{a,r}(T^HM_3,\phi^*g^{TZ_2},h^F)+\TMghc+\mathscr{T}_{\mathscr{H}}\doteq 0.
\end{aligned}\end{align}
By the anomaly formula of Bismut-Lott torsion forms \cite[Theorem 1.5]{Zhu16}, we have
\begin{align}\begin{aligned}\label{e.8}
&\cT_r(T^HM_1,g^{TZ_1},h^F)-\TMgha\\
&\qquad\qquad-\cT_{a,r}(T^HM_3,g^{TZ_3},h^F)+\TMghc+\mathscr{T}_{\mathscr{H}'}\doteq 0,
\end{aligned}\end{align}
where $\mathscr{T}_{\mathscr{H}'}$ is the torsion form associated to
\begin{align}\begin{aligned}\label{e.9}
\mathscr{H}':\quad\cdots \ra H^k(Z_1,Y_{\set{1}};F)\ra H^k(Z_1;F)\oplus H^k(Z_3,Y_{\set{1}};F)\ra H^k(Z_3;F)\ra\cdots.
\end{aligned}\end{align}

\subsection{Computations of $\cT_{a,a/r}(T^H\Xmoo,g^{T\Ymoo},h^F)$ over $X_{[-1,1]}$}

If $\varphi$ is a closed differential form on $Y_{[-1,1]}\cong[-1,1]\times Y$, we have the following decompositon
\begin{align}\begin{aligned}\label{e.241}
\varphi=\varphi_1(u,y)+du\wedge \varphi_2(u,y),\quad \text{where} \quad (u,y)\in [-1,1]\times Y.
\end{aligned}\end{align}
If $\varphi$ is a closed differential form, then we have
\begin{align}\begin{aligned}\label{e.242}
0&=d\varphi=(du\wedge\frac{\partial}{\partial u}+d^Y)(\varphi_1(u,y)+du\wedge \varphi_2(u,y))\\
&=du\wedge\left(\frac{\partial \varphi_1}{\partial u}(u)-d^Y\varphi_2(u)\right)+d^Y\varphi_1(u).
\end{aligned}\end{align}
By Equation \eqref{e.242}, we have
\begin{align}\begin{aligned}\label{e.243}
\varphi_1(1)-\varphi_1(-1)=d^Y\int^1_{-1}\varphi_2(u)du,\quad d^Y(\varphi_1(u))=0,\quad \text{for all} \quad u\in [-1,1].
\end{aligned}\end{align}
We see that $\set{\varphi_1(u):\,u\in[-1,1]}$ gives a family of closed differential forms in the same de Rham cohomology class on $Y$.

We \textbf{assume} that $\varphi$ satisfies the absolute boundary condition at $Y_{\set{-1}}$ and the relative boundary condition at $Y_{\set{1}}$, i.e.,
\begin{align}\begin{aligned}\label{e.244}
\big(i_{\paru}\varphi\big)|_{u=-1}=0,\quad \left(du\wedge\varphi\right)|_{u=1}=0.
\end{aligned}\end{align}

\begin{prop}\label{p.1}
  If we impose the relative boundary conditions at $Y_{\set{1}}$, then the de Rham cohomology groups $H^*(Y_{[-1,1]},Y_{\set{1}};F)$ vanish.
\end{prop}
\begin{proof}
  By \eqref{e.242}-\eqref{e.244}, we have
\begin{align}\begin{aligned}\label{e.245}
\varphi_1(1)=0,\quad \varphi_1(1)-\varphi_1(u)=d^Y\int^1_u\varphi_2(v)dv.
\end{aligned}\end{align}
Hence we get
\begin{align}\begin{aligned}\label{e.246}
\varphi_1(u)=d^Y\int^u_1\varphi_2(v)dv.
\end{aligned}\end{align}
Then by \eqref{e.241}, we have
\begin{align}\begin{aligned}\label{e.247}
\varphi&=d^Y\int^u_1\varphi_2(v)dv+du\wedge\varphi_2(u)\\
&=(d^Y+du\wedge\paru)\int^u_1\varphi_2(v)dv=d\int^u_1\varphi_2(v)dv.
\end{aligned}\end{align}
Set $\phi(u,y):=\int^u_1\varphi_2(v,y)dv$. It is easy to verify that $\phi$ satisfies the absolute boundary condition at $Y_{-1}$ and the relative boundary condition at $Y_1$. By \eqref{e.247}, we get $\varphi=d\phi$, which means that all the closed forms are exact.
\end{proof}

By Definition \ref{d.2} and Proposition \ref{p.1}, we have
\begin{align}\begin{aligned}\label{e.10}
\cT_{a,r}(T^HM_3,g^{TZ_3},h^F):=-\int_0^{+\infty}f^{\wedge}(D_t,h^{\Omega^\bullet(Y_{[-1,1]},F)})_{a,r}\frac{dt}{t},
\end{aligned}\end{align}
and
\begin{align}\begin{aligned}\label{e.11}
\cT_{a,a}(T^HM_3,g^{TZ_3},h^F):=-\int_0^{+\infty}&\left[f^{\wedge}(D_t,h^{\Omega^\bullet(Y_{[-1,1]},F)})_{a,a}-
\frac{\chi'(Y,F)}{2}f'(0)\right.\\
&\left.-\left(\onefour m\rk(F)\chi(Y)-\frac{\chi'(Y,F)}{2}\right)f'(\frac{i\sqrt{t}}{2})\right]\frac{dt}{t}.
\end{aligned}\end{align}
For Equation \eqref{e.11}, we have used the fact that $H^{k}(Y_{[-1,1]};F)\cong H^{k}(Y;F)$ for $0\leq k\leq m-1$.

Now we try to calculate $\cT_{a,a/r}(T^H\Xmoo,g^{T\Ymoo},h^F)$ explicitly.

Let $\ear{t}{x_m}{x'_m}$ be the heat kernel on $[-1,1]$ with the absolute boundary condition at $-1$ and the relative boundary condition at $1$. Let $\eaa{t}{x_m}{x'_m}$ be the heat kernel $-\frac{\partial^2}{\partial x^2_m}$ on $[-1,1]$ with the absolute boundary conditions at $-1,\, 1$. Then we have
\begin{align}\begin{aligned}\label{e.23}
&\eaa{t}{x_m}{x'_m}=\eDD{t}{x_m}{x'_m}\dxm \otimes \dxmu+\eNN{t}{x_m}{x'_m}1 \otimes 1^*,\\
&\ear{t}{x_m}{x'_m}=\eDN{t}{x_m}{x'_m}\dxm \otimes \dxmu+\eND{t}{x_m}{x'_m}1 \otimes 1^*.
\end{aligned}\end{align}
Here the notations $\eDN{t}{\cdot}{\cdot},\cdots $ etc. mean that we impose the Dirichlet boundary condition at $-1$ and
the Neumann boundary condition at $1$.
Let $D^2_t$ be the operator on $X_{[-1,1]}$. Let $\tD_t$ be the operator on the boundary fibration $X$, then by the product structures \eqref{e.1}-\eqref{e.3} we have
\begin{align}\begin{aligned}\label{e.24}
D^2_t=\tD^2_t+\frac{t}{4}\frac{\partial^2}{\partial x^2_m}.
\end{aligned}\end{align}
For their heat kernels, we have the following relations:
\begin{align}\begin{aligned}\label{e.25}
e^{D^2_t}_{a,a}(x,x')=e^{\tD^2_t}(y,y')\otimes \eaa{\frac{t}{4}}{x_m}{x'_m},\quad
e^{D^2_t}_{a,r}(x,x')=e^{\tD^2_t}(y,y')\otimes \ear{\frac{t}{4}}{x_m}{x'_m}.
\end{aligned}\end{align}
By \eqref{e.22}, we have
\begin{align}\begin{aligned}\label{e.26}
&f^{\wedge}(D_t,h^{\Omega^\bullet(Y_{[-1,1]},F)})_{a,a}=\varphi\str\left[\frac{N}{2}f'(D_t)_{a,a}\right]\\
&=\varphi\int_{Y_{[-1,1]}}\tr\left[(-1)^N\frac{N}{2}(1+2D^2_t)e^{D^2_t}_{a,a}(x,x)\right]dv_x\\
&=\frac{\varphi}{2}\int_{Y_{[-1,1]}}\tr\Big[(-1)^N N\big(1+2\tD^2_t+\frac{t}{2}\frac{\partial^2}{\partial x^2_m}\big)e^{\tD^2_t}(y,y)\otimes \eaa{\frac{t}{4}}{x_m}{x_m}\Big]dv_x\\
&=\frac{\varphi}{2}\int_{Y_{[-1,1]}}\tr\Big[(-1)^N N\big(1+2\tD^2_t\big)e^{\tD^2_t}(y,y)\otimes \eaa{\frac{t}{4}}{x_m}{x_m}\Big]dv_x\\
&\qquad+\frac{\varphi}{2}\int_{Y_{[-1,1]}}\tr\Big[(-1)^N Ne^{\tD^2_t}(y,y)\otimes \frac{t}{2}\frac{\partial^2}{\partial x^2_m}\eaa{\frac{t}{4}}{x_m}{x_m}\Big]dv_x\\
&=-\frac{\varphi}{2}\int_{Y_{[-1,1]}}\tr\Big[(-1)^{N_Y} (N_Y+1)f'(\tD_t)(y,y)\otimes \eDD{\frac{t}{4}}{x_m}{x_m}\Big]dv_x\\
&+\frac{\varphi}{2}\int_{Y_{[-1,1]}}\tr\Big[(-1)^{N_Y} N_Yf'(\tD_t)(y,y)\otimes \eNN{\frac{t}{4}}{x_m}{x_m}\Big]dv_x\\
&\qquad-\frac{\varphi}{2}\int_{Y_{[-1,1]}}\tr\Big[(-1)^{N_Y} (N_Y+1)e^{\tD^2_t}(y,y)\otimes \frac{t}{2}\frac{\partial^2}{\partial x^2_m}\eDD{\frac{t}{4}}{x_m}{x_m}\Big]dv_x\\
&\qquad+\frac{\varphi}{2}\int_{Y_{[-1,1]}}\tr\Big[(-1)^{N_Y} N_Ye^{\tD^2_t}(y,y)\otimes \frac{t}{2}\frac{\partial^2}{\partial x^2_m}\eNN{\frac{t}{4}}{x_m}{x_m}\Big]dv_x.
\end{aligned}\end{align}
By the Mckean-Singer formula (cf. \cite[Theorem 3.50]{BGV92}), we have
\begin{align}\begin{aligned}\label{e.27}
1=\chi_{a,a}([-1,1])=\int_{-1}^1\eNN{\frac{t}{4}}{x_m}{x_m}\dxm-\int_{-1}^1\eDD{\frac{t}{4}}{x_m}{x_m}\dxm.
\end{aligned}\end{align}
By \cite[Theorem 3.15]{BL}, \eqref{e.26} and \eqref{e.27}, we have
\begin{align}\begin{aligned}\label{e.28}
f^{\wedge}(D_t,h^{\Omega^\bullet(Y_{[-1,1]},F)})_{a,a}=&\varphi\int_{Y}\str\Big[ \frac{N_Y}{2}f'(\tD_t)(y,y)\Big]dv_y\cdot\chi_{a,a}([-1,1])\\
&\quad-\frac{\varphi}{2}\int_{Y}\str\Big[ f'(\tD_t)(y,y)\Big]dv_y\cdot\int_{-1}^1\eDD{\frac{t}{4}}{x_m}{x_m}\dxm\\
&\quad-\frac{\varphi}{2}\int_{Y}\str\left[e^{\tD^2_t}(y,y)\right]dv_y\cdot\int_{-1}^1\frac{t}{2}\frac{\partial^2}{\partial x^2_m}\eDD{\frac{t}{4}}{x_m}{x_m}\dxm\\
&=f^{\wedge}(\tD_t,h^{\Omega^\bullet(Y,F)})
-\frac{\rk(F)\chi(Y)}{2}\cdot\tr\big[(1+2t\frac{\partial}{\partial t})e^{\frac{t}{4}\frac{\partial^2}{\partial x^2_m}}_{D,D}\big].
\end{aligned}\end{align}
In the same way, one can prove
\begin{align}\begin{aligned}\label{e.29}
f^{\wedge}(D_t,h^{\Omega^\bullet(Y_{[-1,1]},F)})_{a,r}&=
-\frac{\rk(F)\chi(Y)}{2}\cdot\tr\big[(1+2t\frac{\partial}{\partial t})e^{\frac{t}{4}\frac{\partial^2}{\partial x^2_m}}_{D,N}\big].
\end{aligned}\end{align}

\begin{lemma}\label{l.1}
  We have
  \begin{align}\begin{aligned}\label{e.35}
&\int_0^\infty\tr\Big[ (1+2t\frac{\partial}{\partial t})e^{\frac{t}{4}\frac{\partial^2}{\partial x^2_m}}_{D,D}+\onehalf f'(\frac{i\sqrt{t}}{2})\Big]\frac{dt}{t}=-4\log 2.\\
&\int_0^\infty\tr\Big[(1+2t\frac{\partial}{\partial t})e^{\frac{t}{4}\frac{\partial^2}{\partial x^2_m}}_{D,N}\Big]\frac{dt}{t}=-2\log 2.
\end{aligned}\end{align}
\end{lemma}
\begin{proof}
 The zeta functions associated with Laplacians are defined as the Mellin transform of the trace of heat operators \cite[\S 9.6]{BGV92}, i.e.,
\begin{align}\begin{aligned}\label{e.39}
\zeta_{D,D/N}(s):=\frac{1}{\G(s)}\intzi t^s\tr[e^{t\frac{\partial^2}{\partial x^2_m}}_{D,D/N}]\frac{dt}{t}.
\end{aligned}\end{align}

Let $\zeta(s)$ be the Riemann zeta function and $\tzeta(s):=\sum_{k=1}^{\infty}(k-\onehalf)^{-s}$. Then by \cite[\S 9]{Weil99}, one has
\begin{align}\begin{aligned}\label{e.40}
\zeta(0)=-\onehalf,\quad \zeta'(0)=-\onehalf \log(2\pi),\quad \tzeta(0)=0,\quad \tzeta'(0)=-\onehalf \log 2.
\end{aligned}\end{align}
The zeta functions in \eqref{e.39} can also expressed as
\begin{align}\begin{aligned}\label{e.41}
\zDD(s)&=2\left(\frac{2}{\pi}\right)^{2s}\sum_{k=1}^\infty\frac{1}{k^{2s}}
=2\left(\frac{2}{\pi}\right)^{2s}\zeta(2s),\\
\zDN(s)&=2\left(\frac{2}{\pi}\right)^{2s}\sum_{k=1}^\infty\frac{1}{(k-\onehalf)^{2s}}=
2\left(\frac{2}{\pi}\right)^{2s}\tzeta(2s).
\end{aligned}\end{align}
Then we have
\begin{align}\begin{aligned}\label{e.65}
\zDD'(0)&=-4\log 2,\quad \zDN'(0)=-2\log 2.
\end{aligned}\end{align}

Note that $f'(x)=(1+2x^2)e^{x^2}$. Proceeding as in \cite[Theorem 3.29]{BL},
by integration by parts we get that the left hand side of \eqref{e.35} is equal to
\begin{align}\begin{aligned}\label{e.38}
&\int_0^\infty\tr\Big[ (1+2t\frac{\partial}{\partial t})e^{\frac{t}{4}\frac{\partial^2}{\partial x^2_m}}_{D,D}+\onehalf f'(\frac{i\sqrt{t}}{2})\Big]\frac{dt}{t}=\zpDD(0).\\
&\int_0^\infty\tr\Big[(1+2t\frac{\partial}{\partial t})e^{\frac{t}{4}\frac{\partial^2}{\partial x^2_m}}_{D,N}\Big]\frac{dt}{t}=\zpDN(0).
\end{aligned}\end{align}
Then Equation \eqref{e.35} follows from \eqref{e.65} and \eqref{e.38}. The proof is completed.
\end{proof}

\begin{thm}\label{t.5} We have
\begin{align}\begin{aligned}\label{e.64}
\cT_{a,a}(T^H\Xmoo,g^{T\Ymoo},h^F)&=\cTX-2\log 2\rk(F)\chi(Y),\\
\cT_{a,r}(T^H\Xmoo,g^{T\Ymoo},h^F)&=-\log 2\rk(F)\chi(Y).
\end{aligned}\end{align}
\end{thm}

\begin{proof}
By Lemma \ref{l.1}, \eqref{e.10}, \eqref{e.11}, \eqref{e.28} and \eqref{e.29}, we have
\begin{align}\begin{aligned}\label{e.66}
\cT_{a,r}(T^HM_3,g^{TZ_3},h^F)&=\frac{\rk(F)\chi(Y)}{2}\int_0^{+\infty}\cdot\tr\big[f'(\frac{t}{4}\frac{\partial^2}{\partial x^2_m})_{D,N}\big]\frac{dt}{t}\\
&=\frac{\rk(F)\chi(Y)}{2}\cdot (-2\log 2)\\
&=-\log2\rk(F)\chi(Y),
\end{aligned}\end{align}
and
\begin{align}\begin{aligned}\label{e.67}
 \cT_{a,a}(T^HM_3,g^{TZ_3},h^F)
&=\cTX\\
&-\rk(F)\chi(Y)\int_0^{+\infty}\left[-\frac{1}{2}\cdot\tr\big[(1+2t\frac{\partial}{\partial t})e^{\frac{t}{4}\frac{\partial^2}{\partial x^2_m}}_{D,D}\big]
-\onefour f'(\frac{i\sqrt{t}}{2})\right]\frac{dt}{t}\\
&=\cTX+\onehalf\rk(F)\chi(Y)\zDD'(0)\\
&=\cTX-2\log2\rk(F)\chi(Y).
\end{aligned}\end{align}
The proof is completed.
\end{proof}

\textbf{Remark:}\begin{enumerate}
                  \item Moreover, for any $l>0$, we can prove that
\begin{align}\begin{aligned}\label{e.68}
\cT_{a,r}(T^HX_{[-l,l]},g^{TY_{[-l,l]}},h^F)&=-\log 2\rk(F)\chi(Y)\\
\cT_{a,a}(T^HX_{[-l,l]},g^{TY_{[-l,l]}},h^F)&=\cTX-2(\log 2-\frac{\log l}{4})\rk(F)\chi(Y).
\end{aligned}\end{align}
Equation \eqref{e.68} shows that the analytic torsion forms on $X_{[-l,l]}$ with absolute/relative boundary conditions at $\mp 1$ respectively is independent of the length of the cylinder.
                  \item Essentially, Theorem \ref{t.5} can be obtained by the product formula of Bismut-Lott torsion forms \cite[Proposition 3.28]{BL}.
                \end{enumerate}

Then by \eqref{e.8}, \eqref{e.10}, \eqref{e.11}, \eqref{e.28}, \eqref{e.29} and \eqref{e.64}, we have
\begin{align}\begin{aligned}\label{e.30}
&\cT_r(T^HM_1,g^{TZ_1},h^F)-\TMgha\\
&\doteq\cT_{a,r}(T^HM_3,g^{TZ_3},h^F)-\TMghc-\mathscr{T}_{\mathscr{H}'}\\
&\doteq -\cT(T^HX,g^{TY},h^F)-\mathscr{T}_{\mathscr{H}'}+\log 2\rk(F)\chi(Y).
\end{aligned}\end{align}
By Proposition \ref{p.1} and \eqref{e.9}, we have the following commutative diagram of flat vector bundles:
\begin{align}\begin{aligned}\label{e.31}
\xymatrix{
  \cH':\quad\cdots \ar[r]     & H^k(Z_1,Y_{\set{1}};F) \ar[d]^{\Id}\ar[r]& H^k(Z_1;F)\ar[d]^{\Id}\ar[r] &H^k(Y_{[-1,1]};F) \ar[d]^{\alpha_k}\ar[r]&\cdots\\
  \cH'':\quad\cdots \ar[r]       & H^k(Z_1,Y_{\set{1}};F) \ar[r] & H^k(Z_1;F) \ar[r] &H^k(Y;F)\ar[r]&\cdots,}
\end{aligned}\end{align}
where the vertical map $\alpha_k$ is given by $\alpha_k(\phi)=\sqrt{2}\phi$ for $\phi$ being a harmonic form of degree $k$ on $Y$. Note that the maps $\alpha_k$ are isometric with respect to the $L^2-$metrics induced by Hodge theorem.

By \cite[Lemma 2.8]{Zhu15}, the torsion forms associated with $H^k(Y_{[-1,1]};F)\overset{\ak}\lra H^k(Y;F)$ is given by
\begin{align}\begin{aligned}\label{e.32}
\cT\left[0\ra H^k(Y_{[-1,1]};F)\overset{\ak}\lra H^k(Y;F)\ra 0\right]=-\frac{\log 2}{2}\rk(H^k(Y;F)).
\end{aligned}\end{align}
By \cite[Theorem A1.4]{BL}, \eqref{e.31} and \eqref{e.32}, we have
\begin{align}\begin{aligned}\label{e.33}
\cT_{\cH'}-\cT_{\cH''}&\doteq\sum_{k=0}^{m-1}(-1)^k \cT\left[0\ra H^k(Y_{[-1,1]};F)\overset{\ak}\lra H^k(Y;F)\ra 0\right]\\
&\doteq-\frac{\log 2}{2}\sum_{k=0}^{m-1}(-1)^k \rk(H^k(Y;F))\\
&\doteq-\frac{\log 2}{2}\rk(F)\chi(Y).
\end{aligned}\end{align}
By \eqref{e.30} and \eqref{e.33}, we get
\begin{align}\begin{aligned}\label{e.34}
&\cT_r(T^HM_1,g^{TZ_1},h^F)-\TMgha +\cT(T^HX,g^{TY},h^F)+\mathscr{T}_{\mathscr{H}''}\\
&\doteq \log 2\rk(F)\chi(Y)+\frac{\log 2}{2}\rk(F)\chi(Y)=\frac{3}{2}\log 2\rk(F)\chi(Y).
\end{aligned}\end{align}

By \eqref{e.34}, we finally get the comparison formula of the absolute and
relative Bismut-Lott analytic torsion forms.

\begin{thm}\label{t.1}Under the assumptions of product structures \eqref{e.1}-\eqref{e.3}, we have
  \begin{align}\begin{aligned}\label{e.42}
&\cT_r(T^HM_1,g^{TZ_1},h^F)-\TMgha \\
&\qquad\qquad\doteq-\cT(T^HX,g^{TY},h^F)-\mathscr{T}_{\mathscr{H}''}+
 \frac{3}{2}\log 2\rk(F)\chi(Y),
\end{aligned}\end{align}
where $\mathscr{T}_{\mathscr{H}''}$ is the torsion form associated with the second line in the diagram \eqref{e.31}.
\end{thm}

\section{Comparison between the two versions of gluing formulas}\label{s3}

\begin{figure}
 \centering
  \includegraphics[width=10cm]{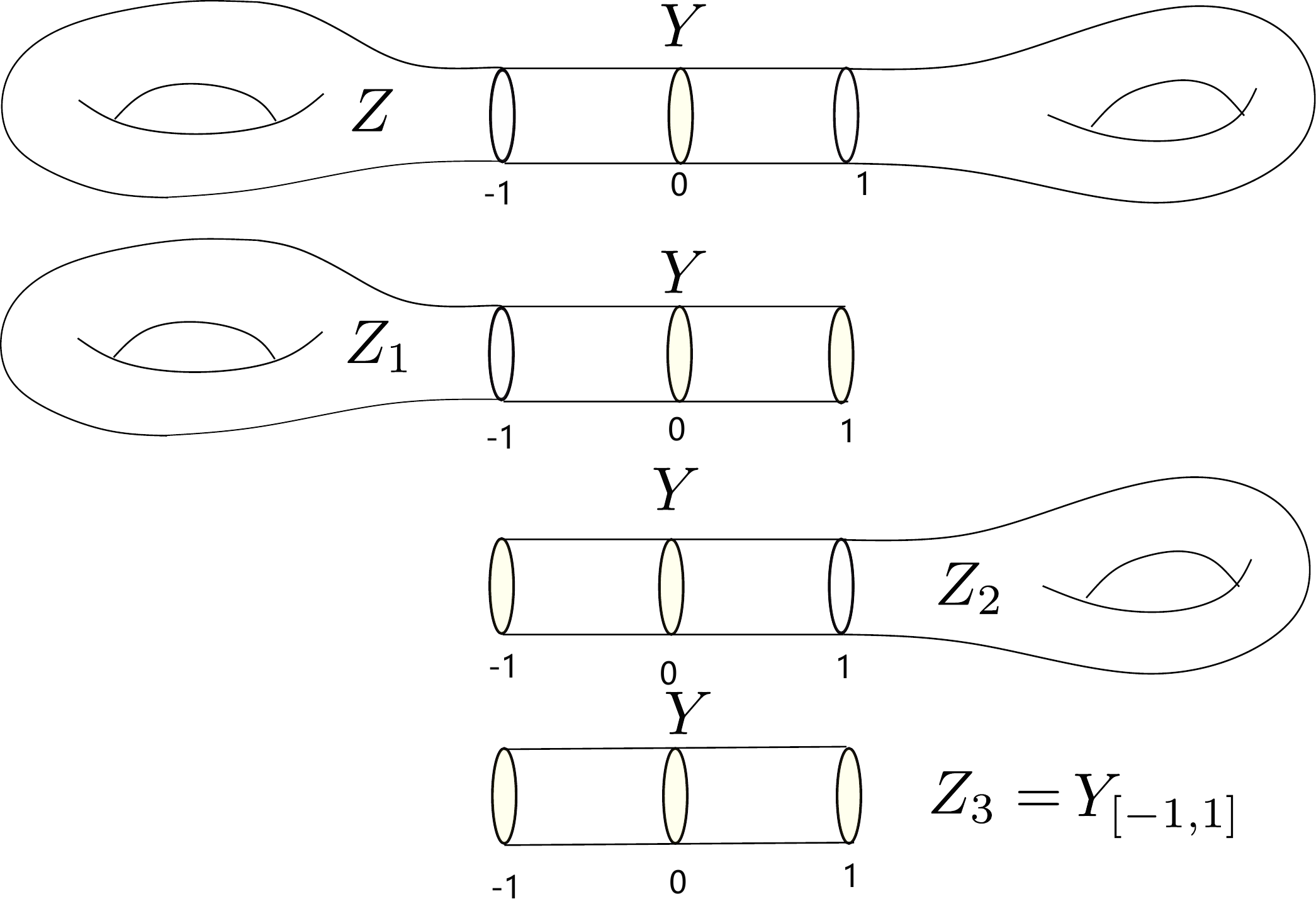}\\
\caption{Gluing relations}\label{figure3}
\end{figure}

Let $M\overset{\pi}\ra S$ be a smooth fibration of standard fiber $Z$. Let $X\subset M$ be a compact surface separating $M$ into two smooth fibrations $M'_1$ and $M'_2$. We assume that $X\ra S$ is also a smooth fibration over $S$ with standard fiber $Y$, such that $Z=Z'_1\cup_Y Z'_2$. Let $X_{[-1,1]}$ be a product neighborhood of $X\subset M$. Set
 \begin{align}\begin{aligned}\label{e.45}
M_1=M'_1\cup_{Y_{\set{0}}}X_{[0,1]},\quad M_2=X_{[-1,0]}\cup_{Y_{\set{0}}}M'_2,\quad M_3=X_{[-1,1]}.
\end{aligned}\end{align}
Then both $M_{i}, i=1,2,3$ have $\Xmoo$ as the cylinder parts, where the product structures of metrics \eqref{e.1}-\eqref{e.3} are assumed. This setting is described in Figure \ref{figure3}.

By the gluing formula of Bismut-Lott torsion forms in \cite{PZZ2}, the following identity holds modulo some exact differential forms over $S$,
\begin{align}\begin{aligned}\label{e.46}
\TMz-\TaMa-&\TaMb\\
&+\TaaMc+\cT_{\cH'}\doteq 0,
\end{aligned}\end{align}
where $\mathscr{T}_{\mathscr{H}'}$ is the torsion form associated to
\begin{align}\begin{aligned}\label{e.47}
\mathscr{H}':\quad\cdots \ra H^k(Z;F)\ra H^k(Z_1;F)\oplus H^k(Z_2;F)\ra H^k(Y_{[-1,1]};F)\ra\cdots.
\end{aligned}\end{align}
Applying the comparison formula in Theorem \ref{t.1} and \eqref{e.33}, we have
\begin{align}\begin{aligned}\label{e.48}
&\TaMa-\cTX\\
&\doteq \TrMa+\cT_{\cH''}-\log2 \rk(F)\chi(Y),
\end{aligned}\end{align}
where $\mathscr{T}_{\mathscr{H}''}$ is the torsion form associated to
\begin{align}\begin{aligned}\label{e.49}
\mathscr{H}'':\quad\cdots \ra H^k(Z_1,Y;F)\ra H^k(Z_1;F) \ra H^k(Y_{[-1,1]};F)\ra\cdots.
\end{aligned}\end{align}
By Theorem \ref{t.5} we have $\TaaMc=\cTX-2\log 2\rk(F)\chi(Y)$, then by \eqref{e.46} and \eqref{e.48} we get
\begin{align}\begin{aligned}\label{e.50}
&\TMz-\TrMa-\TaMb-\cT_{\cH}\\
&\qquad\doteq \cT_{\cH''}-\cT_{\cH'}-\cT_{\cH}+\log 2\rk(F)\chi(Y),
\end{aligned}\end{align}
where $\mathscr{T}_{\mathscr{H}}$ is the torsion form associated to
\begin{align}\begin{aligned}\label{e.51}
\mathscr{H}:\quad\cdots \ra H^k(Z_1,Y;F)\ra H^k(Z;F) \ra H^k(Z_2;F)\ra\cdots.
\end{aligned}\end{align}

\begin{lemma}\label{l.2}The following identity holds modulo some exact differential forms,
\begin{align}\begin{aligned}\label{e.52}
\cT_{\cH'}\doteq \cT_{\cH''}-\cT_{\cH}.
\end{aligned}\end{align}
\end{lemma}
\begin{proof}
The main idea to prove \eqref{e.52} is to decompose the long exact sequences of flat vector bundles into short exact sequences. We essentially follow the techniques used in \cite[\S 3.1]{Ma02}.

We start from the right hand side of \eqref{e.52}. We denote
\begin{align}\begin{aligned}\label{e.53}
\mathscr{H}'':\quad\cdots \rac{\delta''_{k-1}} H^k(Z_1,Y;F)\rac{\alpha''_k} H^k(Z_1;F) \rac{\beta''_k} H^k(Y_{[-1,1]};F)\rac{\delta''_{k}}\cdots.
\end{aligned}\end{align}
\begin{align}\begin{aligned}\label{e.54}
\mathscr{H}:\quad\cdots \rac{\delta_{k-1}} H^k(Z_1,Y;F)\rac{\alpha_k} H^k(Z;F) \rac{\beta_k} H^k(Z_2;F)\rac{\delta_{k}}\cdots.
\end{aligned}\end{align}
We use $(-1)^l\cH''$ to denote the same sequence as $\cH''$ but the degree of each component increased by $l$.
Then we have
\begin{align}\begin{aligned}\label{e.55}
-\mathscr{H}''&:\quad\cdots \rac{\beta''_{k-1}}&H^{k-1}(Y_{[-1,1]};F) &\rac{\delta''_{k-1}}& H^k(Z_1,Y;F)& \rac{\alpha''_k}& H^{k}(Z_1;F)&\rac{\beta''_{k}}\cdots,\\
\mathscr{H}&:\quad\cdots \rac{\delta_{k-1}}& H^k(Z_1,Y;F) &\rac{\alpha_{k}}&H^k(Z;F) & \rac{\beta_k}& H^k(Z_2;F)&\rac{\delta_{k}}\cdots.
\end{aligned}\end{align}
For a long exact sequence $$\cE:\quad 0\ra E^0\rac{f_0} E^1\rac{f_1}\cdots E^k\rac{f_k}E^{k+1}\rac{f_{k+1}}\cdots\rac{f_{n-1}}E^n\ra 0,$$
we set $\cE(k)$ to be its truncation at the degree $k$, i.e.,
\begin{align}\begin{aligned}\label{e.56}
\cE(k):\quad 0\ra E^0\rac{f_0} E^1\rac{f_1}\cdots E^k\rac{f_k}\Im(f_k)\ra 0.
\end{aligned}\end{align}
By using \cite[Lemma 2.7]{Zhu15} and \cite[Lemma 3.1(d)]{Ma02}, we will finish the proof by induction on the degree of truncation.

For $k=1$, on one hand we have
\begin{align}\begin{aligned}\label{e.57}
-\mathscr{H}''(3k-1)&:\quad 0 \ra &0 \qquad&\ra& H^0(Z_1,Y;F)& \rac{\alpha''_0}& H^{0}(Z_1;F)&\rac{\beta''_{0}} \Im(\beta''_{0})\ra 0,\\
\mathscr{H}(3k-1)&:\quad 0 \ra & H^0(Z_1,Y;F) &\rac{\alpha_{0}}&H^0(Z;F) & \rac{\beta_0}& H^0(Z_2;F)&\rac{\delta_{0}}\Im(\delta_{0})\ra 0.
\end{aligned}\end{align}
On the other hand, we have
\begin{align}\begin{aligned}\label{e.58}
-\mathscr{H}'(3k-1)&:\quad 0 \ra & 0 &\ra &H^0(Z;F) & \rac{\alpha'_0}& H^0(Z_1,Y;F)\oplus H^0(Z_2;F)&\rac{\delta'_{0}}\Im(\delta'_{0})\ra 0.
\end{aligned}\end{align}
By \cite[Lemma 3.1(d)]{Ma02}, we have
\begin{align}\begin{aligned}\label{e.59}
\cT(-\mathscr{H}''(2))+\cT(\mathscr{H}(2))=\cT(-\mathscr{H}'(2)).
\end{aligned}\end{align}
In the same way, we can prove
\begin{align}\begin{aligned}\label{e.60}
\cT(-\mathscr{H}''(3k-1))+\cT(\mathscr{H}(3k-1))=\cT(-\mathscr{H}'(3k-1)).
\end{aligned}\end{align}
When $k$ equals $n$ in \eqref{e.60}, we get
\begin{align}\begin{aligned}\label{e.61}
\cT(-\mathscr{H}'')+\cT(\mathscr{H})=\cT(-\mathscr{H}').
\end{aligned}\end{align}
Since we have $\cT((-1)^k\cH)=(-1)^k\cT(\cH)$, Equation \eqref{e.52} follows from \eqref{e.61}.
\end{proof}

 By Lemma \ref{l.2} and \eqref{e.50}, we prove the following theorem.

 \begin{thm}\label{t.4}Under the assumptions of product structures \eqref{e.1}-\eqref{e.3}, we have
   \begin{align}\begin{aligned}\label{e.62}
&\TMz-\TrMa-\TaMb\\
&\qquad\doteq \cT_{\cH} + \log 2\rk(F)\chi(Y),
\end{aligned}\end{align}
where $\mathscr{T}_{\mathscr{H}}$ is the torsion form associated to
\begin{align}\begin{aligned}\label{e.63}
\mathscr{H}:\quad\cdots \ra H^k(Z_1,Y;F)\ra H^k(Z;F) \ra H^k(Z_2;F)\ra\cdots.
\end{aligned}\end{align}
 \end{thm}

\textbf{Remark}: By using the anomaly formula of Bismut-Lott torsion forms \cite[Theorem 1.5]{Zhu16}, Proposition \ref{p.1} and \eqref{e.33}, one can show that the gluing formula in the setting of Theorem \ref{t.4} is equivalent to that of Theorem \ref{t.2} (modulo some exact differential forms).

\bibliographystyle{alpha}

\begin{thebibliography}{PZZ21b}

\bibitem[BG01]{BGo}
J.-M. Bismut and S.~Goette.
\newblock Families torsion and {M}orse functions.
\newblock {\em Ast\'erisque}, (275):x+293, 2001.

\bibitem[BGV04]{BGV92}
N.~Berline, E.~Getzler, and M.~Vergne.
\newblock {\em Heat kernels and {D}irac operators}.
\newblock Grundlehren Text Editions. Springer-Verlag, Berlin, 2004.
\newblock Corrected reprint of the 1992 original.

\bibitem[Bis86]{Bismut86}
J.-M. Bismut.
\newblock The {A}tiyah-{S}inger index theorem for families of {D}irac
  operators: two heat equation proofs.
\newblock {\em Invent. Math.}, 83(1):91--151, 1986.

\bibitem[BL95]{BL}
J.-M. Bismut and J.~Lott.
\newblock Flat vector bundles, direct images and higher real analytic torsion.
\newblock {\em J. Amer. Math. Soc.}, 8(2):291--363, 1995.

\bibitem[BM06]{BruMa06}
J.~Br{\"u}ning and X.~Ma.
\newblock An anomaly formula for {R}ay-{S}inger metrics on manifolds with
  boundary.
\newblock {\em Geom. Funct. Anal.}, 16(4):767--837, 2006.

\bibitem[BM13]{BM13}
J.~Br{\"u}ning and X.~Ma.
\newblock On the gluing formula for the analytic torsion.
\newblock {\em Math. Z.}, 273(3-4):1085--1117, 2013.

\bibitem[BZ92]{BZ92}
J.-M. Bismut and W.~Zhang.
\newblock An extension of a theorem by {C}heeger and {M}\"uller.
\newblock {\em Ast\'erisque}, (205):235, 1992.
\newblock With an appendix by Fran{\c{c}}ois Laudenbach.

\bibitem[BZ94]{BisZh94}
J.-M. Bismut and W.~Zhang.
\newblock Milnor and {R}ay-{S}inger metrics on the equivariant determinant of a
  flat vector bundle.
\newblock {\em Geom. Funct. Anal.}, 4(2):136--212, 1994.

\bibitem[Che79]{Chg}
J.~Cheeger.
\newblock Analytic torsion and the heat equation.
\newblock {\em Ann. of Math. (2)}, 109(2):259--322, 1979.

\bibitem[Has98]{Has}
A.~Hassell.
\newblock Analytic surgery and analytic torsion.
\newblock {\em Comm. Anal. Geom.}, 6(2):255--289, 1998.

\bibitem[Igu02]{Igusa02}
K.~Igusa.
\newblock {\em Higher {F}ranz-{R}eidemeister torsion}, volume~31 of {\em AMS/IP
  Studies in Advanced Mathematics}.
\newblock American Mathematical Society, Providence, RI, 2002.

\bibitem[Igu08]{Igu08}
K.~Igusa.
\newblock Axioms for higher torsion invariants of smooth bundles.
\newblock {\em J. Topol.}, 1(1):159--186, 2008.

\bibitem[LR91]{LoRo91}
J.~Lott and M.~Rothenberg.
\newblock Analytic torsion for group actions.
\newblock {\em J. Differential Geom.}, 34(2):431--481, 1991.

\bibitem[L{\"u}c93]{Luck93}
W.~L{\"u}ck.
\newblock Analytic and topological torsion for manifolds with boundary and
  symmetry.
\newblock {\em J. Differential Geom.}, 37(2):263--322, 1993.

\bibitem[Ma02]{Ma02}
X.~Ma.
\newblock Functoriality of real analytic torsion forms.
\newblock {\em Israel J. Math.}, 131:1--50, 2002.

\bibitem[Mil66]{Mil66}
J.~Milnor.
\newblock Whitehead torsion.
\newblock {\em Bull. Amer. Math. Soc.}, 72:358--426, 1966.

\bibitem[M{\"u}l78]{Mu78}
W.~M{\"u}ller.
\newblock Analytic torsion and {$R$}-torsion of {R}iemannian manifolds.
\newblock {\em Adv. in Math.}, 28(3):233--305, 1978.

\bibitem[M{\"u}l93]{Mu93}
W.~M{\"u}ller.
\newblock Analytic torsion and {$R$}-torsion for unimodular representations.
\newblock {\em J. Amer. Math. Soc.}, 6(3):721--753, 1993.

\bibitem[PZZ20]{PZZ2}
M.~Puchol, Y.~Zhang, and J.~Zhu.
\newblock Adiabatic limit, {W}itten deformation and analytic torsion forms.
\newblock pages 1--76, arXiv:2009.13925, 2020.

\bibitem[PZZ21a]{PZZ21}
M.~Puchol, Y.~Zhang, and J.~Zhu.
\newblock Scattering matrices and analytic torsions.
\newblock {\em Anal. PDE}, 14(1):77--134, 2021.

\bibitem[PZZ21b]{PZZ3}
M.~Puchol, Y.~Zhang, and J.~Zhu.
\newblock A comparison between the {B}ismut-{L}ott torsion and the
  {I}gusa-{K}lein torsion.
\newblock pages 1--30, arXiv:2101.11985v1, 2021.

\bibitem[RS71]{RaySing71}
D.~B. Ray and I.~M. Singer.
\newblock {$R$}-torsion and the {L}aplacian on {R}iemannian manifolds.
\newblock {\em Advances in Math.}, 7:145--210, 1971.

\bibitem[Vis95]{Vishik95}
S.~M. Vishik.
\newblock Generalized {R}ay-{S}inger conjecture. {I}. {A} manifold with a
  smooth boundary.
\newblock {\em Comm. Math. Phys.}, 167(1):1--102, 1995.

\bibitem[Wei99]{Weil99}
A.~Weil.
\newblock {\em Elliptic functions according to {E}isenstein and {K}ronecker}.
\newblock Classics in Mathematics. Springer-Verlag, Berlin, 1999.
\newblock Reprint of the 1976 original.

\bibitem[Zhu15]{Zhu15}
J.~Zhu.
\newblock On the gluing formula of real analytic torsion forms.
\newblock {\em Int. Math. Res. Not. IMRN}, (16):6793--6841, 2015.

\bibitem[Zhu16]{Zhu16}
J.~Zhu.
\newblock Gluing formula of real analytic torsion forms and adiabatic limit.
\newblock {\em Israel J. Math.}, 215(1):181--254, 2016.

\end{thebibliography}

\end{document}